\documentclass[12 pt]{article}%
\usepackage{amsmath, amsfonts, amsthm, color,latexsym}
\usepackage{amsmath}
\usepackage{amsfonts}
\usepackage{amssymb}
\usepackage{color, soul}
\usepackage{graphicx}%
\providecommand{\U}[1]{\protect\rule{.1in}{.1in}}
\allowdisplaybreaks[4]
\newtheorem{theorem}{Theorem}[section]
\newtheorem{proposition}[theorem]{Proposition}
\newtheorem{corollary}[theorem]{Corollary}
\newtheorem{example}[theorem]{Example}

\newtheorem{remark}[theorem]{Remark}

\newtheorem{lemma}[theorem]{Lemma}
\newtheorem{final remark}[theorem]{Final Remark}
\newtheorem{definition}[theorem]{Definition}
\begin{document}

\title{On the transformation of vector-valued sequences by linear and multilinear operators}
\author{Geraldo Botelho\thanks{Supported by CNPq Grant
302177/2011-6 and Fapemig Grant PPM-00326-13.}~ and Jamilson R. Campos\thanks{Supported
by a CAPES Postdoctoral scholarship.\thinspace \hfill\newline\indent2010 Mathematics Subject
Classification: 47L22, 46G25, 46B45.\newline\indent Key words: Banach sequence spaces, ideals of multilinear operators, multilinear stability.}}
\date{}
\maketitle

\begin{abstract} In this paper we provide a unifying approach to the study of Banach ideals of linear and multilinear operators defined, or characterized, by the transformation of vector-valued sequences. We investigate and apply the linear and multilinear stabilities of some frequently used classes of vector-valued sequences.
\end{abstract}

\section*{Introduction}
Important classes of linear and nonlinear operators between Banach spaces are defined, or characterized, by the transformation of vector-valued sequences. Perhaps the most popular example is the celebrated ideal of absolutely $p$-summing operators: a linear operator $u \colon E \longrightarrow F$ is absolutely $p$-summing if $u$ sends weakly $p$-summable sequences in $E$ to absolutely $p$-summable sequences in $F$. The  monograph \cite{djt} is totally devoted to the study of absolutely summing operators. The consideration of these and other classes of vector-valued sequences originated several well studied ideals of linear and multilinear operators -- Banach operator ideals and Banach multi-ideals; which, up to now, have been investigated individually in the literature.
The following examples show just how broad in scope this approach is:\\
(i) $p$-dominated $n$-linear operators send weakly $p$-summable sequences to absolutely $\frac{p}{n}$-summable sequences \cite{boperu,david,popa} (the case $n = 1$ recovers the $p$-summing linear operators),\\
(ii) absolutely $(s;r)$-summing linear and multilinear operators send weakly $r$-summable sequences to absolutely $s$-summable sequences \cite{tiago, BotelhoBlasco, daniellaa},\\
(iii) unconditionally $p$-summing linear and multilinear operators send weakly $p$-summable sequences to unconditionally $p$-summable sequences \cite{matosjunek}, \\
(iv) almost summing linear and multilinear operators send unconditionally summable sequences to almost unconditionally summable sequences \cite{botelhobraunssjunek, danieljoilson, danielmarcela},\\
(v) a multilinear operator is weakly sequentially continuous at the origin if it sends weakly null sequences to norm null sequences \cite{arondineen, livrodineen, valdivia} (the case $n =1$ gives the ideal of completely continuous linear operators).\\
(vi) Cohen strongly $p$-summing multilinear operators send absolutely $np$-summable sequences to Cohen strongly $p$-summable sequences \cite{achour07,jamilson,mezrag09} (the case $n = 1$ recovers the ideal of Cohen strongly $p$-summing linear operators \cite{cohen73}).

The main purpose of this paper is to synthesize the study of these Banach operator ideals and multi-ideals by introducing an abstract framework of generating ideals by means of transformation of vector-valued sequences (cf. Theorem \ref{propideal}) that accomodates the already studied ideals as particular instances (cf. Example \ref{exam}). It is worth mentioning that the notions of {\it finitely determined} and {\it linearly stable} sequences classes we introduce here settle some inaccuracies that have been occurring in the field (cf. Example \ref{erro} and Remark \ref{errodiana}).

It will become clear soon that the preservation of certain classes of vector-valued sequences by linear and/or multilinear operators is an important issue in our subject. As we shall see in Section 2, the preservation by linear operators is a basic assumption to generate ideals and to avoid artificial examples. And, as a matter of fact, all classes of vector-valued sequences that play an important role in the theory of ideals are preserved by linear operators. The results we prove in Section 3 will make clear that the preservation by multilinear operators is a completely different issue. First, not all usual classes of vector-valued sequences are multilinearly stable, for example, weakly $p$-summable sequences and unconditionally $p$-summable sequences for $p > 1$ and weakly null sequences. Second, while the proofs of the linear stability are usually simple, the proofs that some of these classes of sequences are multilinearly stable demand nontrivial typical multilinear arguments. These are the cases of the classes of weakly summable sequences, unconditionally summable sequences, almost unconditionally summable sequences and Cohen strongly $p$-summable sequences. In the end, as an application of our multilinear stability results we improve a result of \cite{bu12} on Cohen almost summing multilinear operators.

The letters $E, E_1, \ldots, E_n,F$ shall denote Banach spaces over $\mathbb{K} = \mathbb{R}$ or $\mathbb{C}$. The closed unit ball of $E$ is denoted by $B_E$ and its topological dual by $E'$. By $BAN$ we denote the class of all Banach spaces over $\mathbb{K}$. Given Banach spaces $E$ and $F$, the symbol $E \stackrel{1}{\hookrightarrow} F$ means that $E$ is a linear subspace of $F$ and $\|x\|_F \leq \|x\|_E$ for every $x \in E$. By ${\cal L}(E_1, \ldots, E_n;F)$ we denote the Banach space of $n$-linear operators $A \colon E_1 \times \cdots \times E_n \longrightarrow F$ endowed with the usual sup norm. Given $\varphi_m \in E_m'$, $m = 1, \ldots, n$, and $b \in F$, consider the operator $\varphi_1 \otimes \cdots \otimes \varphi_n \in {\cal L}(E_1, \ldots, E_n;F)$ given by
$$\varphi_1 \otimes \cdots \otimes \varphi_n \otimes b(x_1, \ldots, x_n) = \varphi_1(x_1) \cdots \varphi_n(x_n)b. $$
Linear combinations of operators of this type are called {\it $n$-linear operators of finite type}.

\section{Classes of vector-valued sequences}

In this section we construct an abstract framework that encompasses all classes of vector-valued sequences that are important in our study. By $c_{00}(E)$ we denote the set of all $E$-valued finite sequences, which, as usual, can be regarded as infinite sequences by completing with zeros. For every $j \in \mathbb{N}$, $e_j = (0,\ldots, 0,1,0,0,\ldots)$ where $1$ appears at the $j$-th coordinate.

\begin{definition}\label{def1}\rm A {\it  class of vector-valued sequences $X$}, or simply a {\it sequence class $X$}, is a rule that assigns to each $E \in BAN$ a Banach space $X(E)$ of $E$-valued sequences, that is $X(E)$ is a vector subspace of $E^{\mathbb{N}}$ with the coordinatewise operations, such that:
 $$c_{00}(E) \subseteq X(E) \stackrel{1}{\hookrightarrow}  \ell_\infty(E){\rm ~ and~ } \|e_j\|_{X(\mathbb{K})}= 1 {\rm ~for~ every~} j.$$

A sequence class $X$ is {\it finitely determined} if for every sequence $(x_j)_{j=1}^\infty \in E^{\mathbb{N}}$, $(x_j)_{j=1}^\infty \in X(E)$ if and only if $\displaystyle\sup_k \left\|(x_j)_{j=1}^k  \right\|_{X(E)} < +\infty$ and, in this case,
$$\left\|(x_j)_{j=1}^\infty  \right\|_{X(E)} = \sup_k \left\|(x_j)_{j=1}^k  \right\|_{X(E)}. $$
\end{definition}

\begin{example}\label{firstex}\rm Let $1 \leq p < + \infty$. Letting $X(E)$ be any of the spaces listed below, the rule $E \mapsto X(E)$ is a sequence class:

\medskip

\noindent $\bullet$ $\ell_\infty(E)$ = bounded $E$-valued sequences with the sup norm.\\
$\bullet$ $c_0(E)$ = norm null $E$-valued sequences with the sup norm.\\
$\bullet$ $c_0^w(E)$ = weakly null $E$-valued sequences with the sup norm.\\
$\bullet$ $\ell_p(E) $ = absolutely $p$-summable $E$-valued sequences with the usual norm $\|\cdot\|_p$.\\
$\bullet$ $\ell_p^w(E)$ = weakly $p$-summable $E$-valued sequences with the norm
$$\|(x_j)_{j=1}^\infty\|_{w,p} = \sup_{\varphi \in B_{E'}}\|(\varphi(x_j))_{j=1}^\infty\|_p. $$
$\bullet$ $\ell_p^u(E) = \left\{(x_j)_{j=1}^\infty \in \ell_p^w(E) : \displaystyle\lim_k \|(x_j)_{j=k}^\infty\|_{w,p} = 0 \right\}$ with the norm inherited from $\ell_p^w(E)$ (unconditionally $p$-summable sequences, see \cite[8.2]{df}).\\
$\bullet$ ${\rm Rad}(E)$ = almost unconditionally summable $E$-valued sequences, in the sense of \cite[Chapter 12]{djt}, with the norm $\displaystyle\|(x_j)_{j=1}^\infty\|_{\rm Rad(E)} = \left(\int_0^1 \left\|\sum_{j=1}^\infty r_j(t) x_j \right\|^2 dt \right)^{1/2},$
where $(r_j)_{j=1}^\infty$ are the Rademacher functions.\\
$\bullet$ $\displaystyle{{\rm RAD}(E)} = \left\{(x_j)_{j=1}^\infty \in E^{\mathbb{N}}: \|(x_j)_{j=1}^\infty\|_{{\rm RAD}(E)}:= \sup_k \|(x_j)_{j=1}^k\|_{{\rm Rad}(E)}< +\infty\right\}$ \cite{BotelhoBlasco, blascoetal}.\\
$\bullet$ $\displaystyle\ell_p\langle E \rangle  = \left\{(x_j)_{j=1}^\infty \in E^{\mathbb{N}}: \|(x_j)_{j=1}^\infty\|_{\ell_p\langle E \rangle}:= \sup_{(\varphi_j)_{j=1}^\infty \in B_{\ell_{p^*}^w(E')}} \|(\varphi_j(x_j))_{j=1}^\infty\|_1< +\infty\right\}$,\\ where $\frac{1}{p} + \frac{1}{p^*} = 1$,  (Cohen strongly $p$-summable sequences, see, e.g., \cite{cohen73}).

\medskip

The sequence classes $\ell_\infty(\cdot), \ell_p(\cdot), \ell_p^w(\cdot)$, $\ell_p\langle \,\cdot\, \rangle$ and ${\rm RAD}(\cdot)$ are finitely determined.
\end{example}

\begin{remark}\label{remark1}\rm (a) A few words about the spaces ${\rm Rad}(E)$ and ${\rm RAD}(E)$ are in order:  for every Banach space $E$, ${\rm Rad}(E) \subseteq {\rm RAD}(E)$ with equality of norms in ${\rm Rad}(E)$; and ${\rm Rad}(E) = {\rm RAD}(E)$ if and only if $E$ does not contain a copy of $c_0$ (see, e.g., \cite[Section V.5]{tarieladze}).\\
(b) The theory could be alternatively developed by letting $X(E)$ be a $p$-Banach space, $0 < p \leq 1$, keeping the same $p$ for every Banach space $E$. In this case, the framework we are constructing would encompass the spaces $\ell_p(E), \ell_p^w(E)$ and $\ell_p^u(E)$ for $0 < p < 1$.
\end{remark}

The following result describes how the transformation of vector-valued sequences by linear and multilinear operators works. In what follows, for the linear case just take $n=1$.

\begin{proposition}\label{firstprop} Let $n \in \mathbb{N}$ and $X_1, \ldots, X_n,Y$ be sequence classes. The following conditions are equivalent for a given multilinear operator $A \in {\cal L}(E_1, \ldots, E_n;F)$:\\
{\rm (a)} $(A(x_j^1, \ldots, x_j^n))_{j=1}^\infty \in Y(F)$ whenever $(x_j^m)_{j=1}^\infty \in X_m(E_m)$, $m = 1, \ldots, n$.\\
{\rm (b)} The induced map
$$\widehat{A}\colon X_1(E_1) \times \cdots \times X_n(E_n) \longrightarrow Y(F)~,~\widehat{A}\left ((x_j^1)_{j=1}^\infty, \ldots, (x_j^n)_{j=1}^\infty\right)= (A(x_j^1, \ldots, x_j^n))_{j=1}^\infty, $$
is a well-defined continuous $n$-linear operator.

The conditions above imply condition {\rm (c)} below, and they are all equivalent if the sequence classes $X_1, \ldots, X_n$ and $Y$ are finitely determined.\\
{\rm (c)} There is a constant $C> 0$ such that
\begin{equation}
\left\|(A(x_j^1, \ldots, x_j^n))_{j=1}^k\right\|_{Y(F)}\leq C \cdot \prod_{m=1}^n \left\|(x_j^m)_{j=1}^k\right\|_{X_m(E_m)}, \label{eq1}
\end{equation}
for every $k \in \mathbb{N}$ and all finite sequences $x_j^m \in E_m, j=1, \ldots, k, m = 1,\ldots, n$.

In this case,
\begin{equation}\|\widehat{A}\|= \inf\{C : (\ref{eq1}) {\rm ~holds}\}\label{eqnorma} .\end{equation}
\end{proposition}

\begin{proof} The implication (b) $\Longrightarrow$ (a) is immediate. Assuming (a) it is clear that $\widehat{A}$ is well-defined and $n$-linear. Let us prove that it is continuous in the case $n = 2$ (the general case is identical). To do so, let $(x_j)_{j=1}^\infty \in X_1(E_1)$ and  $(y_j)_{j=1}^\infty \in X_2(E_2)$ be sequences such that $(x_j,y_j) \longrightarrow (x,y)$ in $X_1(E_1) \times X_2(E_2)$ and $\widehat{A}(x_j, y_j) \longrightarrow z$ in $Y(G)$. Then $x_j \longrightarrow x$ in $X_1(E_1)$ and $y_j \longrightarrow y$ in $X_2(E_2)$. Call
$$x_j = (\xi_{j,m})_{m=1}^\infty,\, y_j = (\eta_{j,m})_{m=1}^\infty,\, x = (\xi_m)_{m=1}^\infty,\, y = (\eta_m)_{m=1}^\infty,\, z = (w_m)_{m=1}^\infty. $$
 The condition $X(\cdot) \stackrel{1}{\hookrightarrow}  \ell_\infty(\cdot)$ guarantees that convergence in the sequence spaces we working with implies coordinatewise convergence, so $\xi_{j,m} \stackrel{j}{\longrightarrow} \xi_m$ in $E_1$ and $ \eta_{j,m} \stackrel{j}{\longrightarrow} \eta_m$ in $E_2$ for every $m$. The continuity of $A$ gives $A(\xi_{j,m}, \eta_{j,m}) \stackrel{j}{\longrightarrow} A(\xi_m, \eta_m)$ in $F$ for every $m$. From
$$\left(A(\xi_{j,m}, \eta_{j,m})  \right)_{m=1}^\infty = \widehat{A} (x_j, y_j) \stackrel{j}{\longrightarrow} z = (w_m)_{m=1}^\infty {\rm ~in~} Y(F), $$
it follows that $A(\xi_{j,m}, \eta_{j,m})\stackrel{j}{\longrightarrow} w_m$ in $F$ for every $m$. Hence $A(\xi_m, \eta_m) = w_m$ for every $m$. Finally,
$$\widehat{A}(x,y) = \widehat{A} \left((\xi_m)_{m=1}^\infty,(\eta_m)_{m=1}^\infty \right)= \left(A(\xi_m, \eta_m) \right)_{m=1}^\infty = (w_m)_{m=1}^\infty = z,$$
proving that the graph of $\widehat{A}$ is closed. The continuity of $\widehat{A}$ follows from the closed graph theorem for multilinear operators (see, e.g., \cite{cecilia}).

The implication (b) $\Longrightarrow$ (c) and the inequality $ \inf\{C : (\ref{eq1}) {\rm ~holds}\} \leq \|\widehat{A}\|$ are obvious. Supposing that $X_1, \ldots, X_n$ and $Y$ are finitely determined, (c) $\Longrightarrow$ (b) and the reverse inequality  follow by taking the supremum over $k$ in (\ref{eq1}).
\end{proof}

In daily life, it is more practical do handle inequality (\ref{eq1}) rather than the transformation of vector-valued sequences. Problems, like the one described in Example \ref{erro}, may occur when the sequences classes involved are not finitely determined. Fortunately, as we shall see now, sometimes the transformation of non-finitely determined sequence classes is equivalent to the transformation of some other finitely determined ones. The consequence of Proposition \ref{firstprop} we shall present next, though a bit technical, has several practical applications (cf. Example \ref{erro} and the proofs of Theorems \ref{firsttheo} and \ref{theorad}).

When we say that a Banach space $E$ is a closed subspace of the Banach space $F$ we mean that $E$ is a linear subspace of $F$ and $\|\cdot\|_E = \|\cdot\|_F$ in $E$.

\begin{definition}\rm Let $X$ and $Y$ be sequence classes. We say that:\\
$\bullet$ $X < Y$ if, for every Banach space $E$, $X(E)$ is a closed subspace of $Y(E)$ and, for every sequence $(x_j)_{j=1}^\infty \in Y(E)$, $(x_j)_{j=1}^\infty \in X(E)$ if and only if $\displaystyle\lim_k \|(x_j)_{j=k}^\infty\|_{Y(E)} = 0$.\\
$\bullet$ $X \prec Y$ if, for every Banach space $E$, $X(E)$ is a closed subspace of $Y(E)$ and, for every sequence $(x_j)_{j=1}^\infty \in Y(E)$, $(x_j)_{j=1}^\infty \in X(E)$ if and only if $\displaystyle\lim_{k,l} \|(x_j)_{j=k}^l\|_{Y(E)} = 0$.
\end{definition}

\begin{corollary}\label{cornew} Let $n \in \mathbb{N}$ and $X_1, \ldots, X_n,Y, Z_1, \ldots, Z_n, W$ be sequence classes such that $Z_1, \ldots, Z_n, W$ are finitely determined. Suppose one of the following conditions holds:\\
\indent {\rm (i)} $X_m < Z_m$ for $m = 1, \ldots, n$, and $Y< W$;\\
\indent {\rm (ii)} $X_m \prec Z_m$ for $m = 1, \ldots, n$, and $Y \prec W$;\\
\indent {\rm (iii)} $X_m < Z_m$ for $m = 1, \ldots, n$, and $Y \prec W$.\\
Then, for all Banach spaces $E_1, \ldots, E_n, F$, the following are equivalent for an $n$-linear operator $A \in {\cal L}(E_1, \ldots, E_n;F)$:\\
{\rm (a)} $(A(x_j^1, \ldots, x_j^n))_{j=1}^\infty \in Y(F)$ whenever $(x_j^m)_{j=1}^\infty \in X_m(E_m)$, $m = 1, \ldots, n$.\\
{\rm (b)} $(A(x_j^1, \ldots, x_j^n))_{j=1}^\infty \in W(F)$ whenever $(x_j^m)_{j=1}^\infty \in Z_m(E_m)$, $m = 1, \ldots, n$.

In this case,
$$\|\widehat{A}\colon X_1(E_1) \times \cdots \times X_n(E_n) \longrightarrow Y(F)\| = \| \widehat{A}\colon Z_1(E_1) \times \cdots \times Z_n(E_n) \longrightarrow W(F)\|.$$
\end{corollary}

\begin{proof} (a) $\Longrightarrow$ (b) From (a) and Proposition \ref{firstprop}[(a)  $\Longrightarrow$ (c)], having in mind that, in all three cases, $X_m(E_m)$ is a closed subspace of $Z_m(E_m)$ and $Y(F)$ is a subspace of $W(F)$), it follows that $$\left\|(A(x_j^1, \ldots, x_j^n))_{j=1}^k\right\|_{W(F)}\leq C \cdot \prod_{m=1}^n \left\|(x_j^m)_{j=1}^k\right\|_{Z_m(E_m)},$$
for arbitrary finite sequences. As $Z_1, \ldots, Z_n, W$ are finitely determined, Proposition \ref{firstprop}[(c)  $\Longrightarrow$ (a)] gives (b).\\
(b) $\Longrightarrow$ (a) Again from Proposition \ref{firstprop}[(a)  $\Longrightarrow$ (c)] we have
\begin{equation}\left\|(A(x_j^1, \ldots, x_j^n))_{j=k}^l\right\|_{Y(F)}\leq C \cdot \prod_{m=1}^n \left\|(x_j^m)_{j=k}^l\right\|_{X_m(E_m)},
\label{eqeq}
\end{equation}
for arbitrary $l > k$ and vectors $x_j^m \in E_m$. Let $(x_j^m)_{j=1}^\infty \in X_m(E_m)$, $m = 1, \ldots, n$, be given. By (b) we have $(A(x_j^1, \ldots, x_j^n))_{j=1}^\infty \in W(F)$.\\
(i) For $X_m < Z_m$ we have $\displaystyle\lim_k \|(x_j^m)_{j=k}^\infty\|_{X_m(E)} = 0$, $m = 1, \ldots, n$. Taking the supremum over $l > k$ for a fixed $k$ in (\ref{eqeq}), as each $Z_m$ and $W$ are finitely determined we obtain
\begin{align*}\left\|(A(x_j^1, \ldots, x_j^n))_{j=k}^\infty\right\|_{W(F)} & \leq C \cdot \prod_{m=1}^n \sup_l\left\|(x_j^m)_{j=k}^l\right\|_{Z_m(E_m)} =  C \cdot \prod_{m=1}^n \left\|(x_j^m)_{j=k}^\infty\right\|_{Z_m(E_m)},
\end{align*}
for every $k$. Taking now the limit for $k \longrightarrow +\infty$ we get  $\displaystyle\lim_k \|(A(x_j^1, \ldots, x_j^n))_{j=k}^\infty\|_{W(F)} = 0$. Since $Y < W$, it follows that $(A(x_j^1, \ldots, x_j^n))_{j=1}^\infty \in Y(F)$.\\
(ii) For $X_m \prec Z_m$ we have $\displaystyle\lim_{k,l} \|(x_j^m)_{j=k}^l\|_{X_m(E)} = 0$, $m = 1, \ldots, n$. Taking the limit for $k,l \longrightarrow + \infty$ in (\ref{eqeq}) we conclude that $\displaystyle\lim_{k,l} \|(A(x_j^1, \ldots, x_j^n))_{j=k}^l\|_{W(F)} = 0$. Since $Y \prec W$, it follows that $(A(x_j^1, \ldots, x_j^n))_{j=1}^\infty \in Y(F)$.\\
(iii) As $X_m < Z_m$ for $m = 1, \ldots, n$, and each $Z_m$ is finitely determined, it is easy to see that taking the limit for $k,l \longrightarrow + \infty$ in (\ref{eqeq}) the conclusion follows as in (ii).

The equality of the norms is immediate from (\ref{eqnorma}).
\end{proof}

\begin{example}\rm \label{erro} According to \cite[p.\,234]{djt}, the following conditions are equivalent for a linear operator $u \in {\cal L}(E;F)$ (the operator is called {\it almost summing} in this case):\\
(a) $(u(x_j))_{j=1}^\infty \in {\rm Rad}(F)$ whenever $(x_j)_{j=1}^\infty \in \ell_2^w(E)$.\\
(b) There is a constant $C> 0$ such that
\begin{equation}
\left\|(u(x_j))_{j=1}^k\right\|_{{\rm Rad}(F)}\leq C \cdot \left\|(x_j)_{j=1}^k\right\|_{w,2},\label{almostsumming}
\end{equation}
for all finite sequences $x_1, \ldots, x_k \in E$.

As noted in \cite{nachr}, these conditions are {\it not} always equivalent. They are equivalent if $F$ does not contain a copy of $c_0$ (cf. Remark \ref{remark1}(a)).  The point is that ${\rm Rad}(\cdot)$ is not finitely determined. This imprecision has caused a lot of trouble in the study of linear and non-linear almost summing operators, e.g., the part on almost summing operators in \cite[Section 4]{diana} is not correct.  As usual, it is the inequality that is used in the computations, so it has been taken for grant that almost summing operators are defined by inequality (\ref{almostsumming}). Following this understanding, the classes of almost summing linear and multilinear operators and their relatives have been extensively studied (see, e.g., \cite{nachr, botelhobraunssjunek, bukranz, danieljoilson, danielmarcela} and a very recent contribution in \cite{popaarchivnovo}) using inequality (\ref{almostsumming}) or its multilinear version as definition.
%
As to the transformation of vector-valued sequences, since $\ell_2^u(\cdot) < \ell_2^w(\cdot)$ and ${\rm Rad}(\cdot) \prec {\rm RAD}(\cdot)$, by Corollary \ref{cornew}(iii) the following are equivalent for an $n$-linear operator $A \in  {\cal L}(E_1, \ldots, E_n;F)$:\\
(i) $A$ is almost summing.\\
(ii) $(A(x_j^1, \ldots, x_j^n))_{j=1}^\infty \in {\rm Rad}(F)$ whenever $(x_j^m)_{j=1}^\infty \in \ell_2^u(E_m)$, $m = 1, \ldots, n$.\\
(iii) $(A(x_j^1, \ldots, x_j^n))_{j=1}^\infty \in {\rm RAD}(F)$ whenever $(x_j^m)_{j=1}^\infty \in \ell_2^w(E_m)$, $m = 1, \ldots, n$.

In particular, the case $n = 1$ settles the imprecision in the original definition of almost summing linear operators. Observe that $\ell_2^w(\cdot)$ and ${\rm RAD}(\cdot)$ are finitely determined whereas $\ell_2^u(\cdot)$ and ${\rm Rad}(\cdot)$ are not. The equivalence (i) $\Leftrightarrow$ (ii) goes back to \cite[Theorem 3.1]{nachr} and \cite[Theorem 3.3]{botelhobraunssjunek}.
\end{example}

\section{Banach multi-ideals}
In this section we study the classes of linear and multilinear operators satisfying the equivalent conditions of Proposition \ref{firstprop}. As before, for the linear case just take $n = 1$.

\begin{definition}\rm Let $n \in \mathbb{N}$ and $X_1, \ldots, X_n,Y$ be sequence classes. A multilinear operator $A \in {\cal L}(E_1, \ldots, E_n;F)$ is {\it $(X_1,\ldots, X_n;Y)$-summing} if the equivalent conditions of Proposition \ref{firstprop} hold for $A$, that is, $(A(x_j^1, \ldots, x_j^n))_{j=1}^\infty \in Y(F)$ whenever $(x_j^m)_{j=1}^\infty \in X_m(E_m)$, $m = 1, \ldots, n$. In this case we write $A \in {\cal L}_{X_1,\ldots,X_n;Y}(E_1, \ldots, E_n;F)$ and define
$$\|A\|_{X_1, \ldots, X_n;Y} = \|\widehat{A}\|_{{\cal L}(X_1(E_1), \ldots, X_n(E_n); Y(F))}.$$
If $X_1 = \cdots =X_n=X$ we simply write ${\cal L}_{n,X;Y}$ and $\|\cdot\|_{n,X;Y}$. In the linear case, that is $n=1$, we write ${\cal L}_{X;Y}$ and $\|\cdot\|_{X;Y}$.  
\end{definition}

By ${\cal L}^n$ we mean the class of all continuous $n$-linear operators with the usual sup norm. For $\frac{1}{p} \leq \frac{1}{p_1} + \cdots + \frac{1}{p_n}$, H\"older's inequality gives
\begin{equation}{\cal L}_{\ell_{p_1}(\cdot),\ldots, \ell_{p_n}(\cdot); \ell_p(\cdot) } \stackrel{1}{=} {\cal L}^n,\label{holder}
\end{equation}
where $\stackrel{1}{=}$ means equality of norms.



The following concept is crucial for ${\cal L}_{X_1,\ldots,X_n;Y}$ to be an operator ideal or a multi-ideal. 

\begin{definition}\rm A sequence class $X$ is said to be {\it linearly stable} if ${\cal L}_{X;X}(E;F) \stackrel{1}{=} {\cal L}(E;F)$ for all Banach spaces $E$ and $F$, that is, for every $u \in {\cal L}(E;F)$, $(u(x_j))_{j=1}^\infty \in X(F)$ whenever $(x_j)_{j=1}^\infty \in X(E)$ and $\|\widehat{u}\colon X(E) \longrightarrow X(F)\| = \|u\|  $.
\end{definition}

\begin{example}\rm
(a) All sequence classes listed in Example \ref{firstex} are linearly stable.\\
(b) The sequence class
$$E \in BAN \mapsto X(E) =\left\{\begin{array}{cl} c_0(E), & {\rm ~if~} E {\rm~is~reflexive} \\ \ell_\infty(E), & {\rm ~otherwise},  \end{array} \right. $$
fails to be linearly stable (consider, for example, the inclusion $\ell_1 \hookrightarrow \ell_2$). Linear stability has the extra purpose of avoiding such artificial constructions.
\end{example}

The following (expected) properties shall be useful later:

\begin{lemma}\label{lema1}Let $n \in \mathbb{N}$ and $X, X_1, \ldots, X_n,Y$ be  linearly stable sequence classes.\\
{\rm (a)} For every Banach space $E$,
$$\|(0, \ldots, 0,x,0,0,0, \ldots)\|_{X(E)} = \|x\|_E$$ regardless of the vector $x \in E$ and the position $x$ appears in the sequence.\\
{\rm (b)} For $k \in \mathbb{N}$ and sequences $(x_1^m)_{m=1}^\infty, \ldots, (x_k^m)_{m=1}^\infty$ in $E$, if $\displaystyle\lim_m x_j^m = x_j \in E$ for $j=1,\ldots,k$, then $$\lim_m \,(x_1^m,\ldots,x_k^m,0,0,\ldots)= (x_1,\ldots,x_k,0,0,\ldots)\  {\rm in}\  X(E).$$
{\rm (c)} $\|A\| \leq \|A\|_{X_1,\ldots, X_n;Y}$ for every $A \in {\cal L}_{X_1, \ldots, X_n;Y}(E_1, \ldots, E_n;F)$.
\end{lemma}

\begin{proof} (a) Consider the continuous linear operator
$$u_x \colon \mathbb{K} \longrightarrow E~,~u_x(\lambda) = \lambda x. $$
Letting $j \in \mathbb{N}$ be such that $x$ appears in the $j$-th position in the sequence, 
\begin{align*} \|x\|_E \leq \|(0, \ldots, 0,x,0,0, \ldots )\|_{X(E)} &= \|\widehat{u_x}(e_j)\|_{X(E)}
\leq \|\widehat{u_x}\|_{{\cal L}(X(\mathbb{K}), X(E))}\cdot \|e_j\|_{X(\mathbb{K})}\\&= \|u_x\|_{{\cal L}(\mathbb{K},E)}= \|x\|_E,
\end{align*}
where the first inequality is a consequence of the condition $X(E) \stackrel{1}{\hookrightarrow}  \ell_\infty(E)$.\\
(b) From (a),
\begin{align*} (x_1,\ldots,x_k,0,0,\ldots) &= \sum_{j=1}^k (0,\ldots,0,x_j,0,\ldots) \stackrel{(a)}{=} \sum_{j=1}^k \lim_m\,(0,\ldots,0,x_j^m,0,\ldots)\\
& =  \lim_m\sum_{j=1}^k (0,\ldots,0,x_j^m,0,\ldots) = \lim_m \,(x_1^m,\ldots,x_k^m,0,0,\ldots).
\end{align*}
(c) For all $x_1 \in E_1, \ldots, x_n \in E_n$, by (a) we have
\begin{align*} \|A(x_1, \ldots,& x_n)\|_F = \|(A(x_1, \ldots, x_n), 0,0, \ldots)\|_{Y(F)}\\
&= \|\widehat{A}((x_1,0,0,\ldots),\ldots, (x_n,0,0,\ldots))\|_{Y(F)} \\
&\leq \|\widehat{A}\|_{{\cal L}(X_1(E_1), \ldots, X_n(E_n); Y(F))}\cdot \|(x_1, 0,0,\ldots)\|_{X_1(E_1)} \cdots  \|(x_n, 0,0,\ldots)\|_{X_n(E_n)}\\
& = \|A\|_{X_1,\ldots, X_n;Y}\cdot \|x_1\|_{E_1} \cdots \|x_n\|_{E_n},
\end{align*}
from which the desired inequality follows.
\end{proof}

The notion of operator ideals was systematized by Pietsch \cite{livropietsch} and  multi-ideals were introduced by Pietsch \cite{pietsch} and has been developed by several authors (for recent contributions, see, e.g., \cite{jamilson, galicer, danieljoilson, danieljoilsonmonat, popa, diana}).

\begin{definition}\rm Let $n\in \mathbb{N}$. A {\it Banach ideal of $n$-linear operators} is a pair $({\cal M}_n, \|\cdot\|_{{\cal M}_n})$ where ${\cal M}_n$ is as subclass of the class of all $n$-linear operators between Banach spaces and $\|\cdot\|_{{\cal M}_n} \colon {\cal M}_n\longrightarrow \mathbb{R}$ is a function such that, for all Banach spaces $E_1, \ldots, E_n,F$, the component
$${\cal M}_n(E_1, \ldots, E_n;F) := {\cal L}(E_1, \ldots, E_n;F)\cap {\cal M}_n $$
is a linear subspace of ${\cal L}(E_1, \ldots, E_n;F)$ on which $\|\cdot\|_{{\cal M}_n}$ is a complete norm and\\
(i) ${\cal M}_n(E_1, \ldots, E_n;F)$ contains the $n$-linear operators of finite type and $$\left\|I_n \colon \mathbb{K}^n \longrightarrow \mathbb{K}~,~I_n(\lambda_1, \ldots, \lambda_n) = \lambda_1 \cdots \lambda_n \right\|_{{\cal M}_n} =1. $$
(ii) If $A \in {\cal M}_n(E_1, \ldots, E_n;F)$, $u_m \in {\cal L}(G_m,E_m)$, $m = 1, \ldots, n$, and $v \in {\cal L}(F;H)$, then $v \circ A \circ (u_1, \ldots, u_n) \in {\cal M}_n(G_1, \ldots, G_n;H)$ and
$$\|v \circ A \circ (u_1, \ldots, u_n)\|_{{\cal M}_n} \leq \|v\| \cdot \|A\|_{{\cal M}_n} \cdot \|u_1\| \cdots \|u_n\|.  $$
The case $n = 1$ gives the classical notion of Banach operator ideals.
\end{definition}

Given sequence classes $X_1, \ldots, X_n, Y$, we say that $X_1(\mathbb{K}) \cdots X_n(\mathbb{K}) \stackrel{1}{\hookrightarrow} Y(\mathbb{K})$ if \\$(\lambda_j^1 \cdots \lambda_j^n)_{j=1}^\infty \in Y(\mathbb{K})$ and
$$\left\|(\lambda_j^1 \cdots \lambda_j^n)_{j=1}^\infty \right\|_{Y(\mathbb{K})}\leq  \prod_{m=1}^n \left\|(\lambda_j^m)_{j=1}^\infty \right\|_{X_m(\mathbb{K})}, $$
whenever $(\lambda_j^m)_{j=1}^\infty \in X_m(\mathbb{K}), m = 1, \ldots, n.$

\begin{theorem}\label{propideal} Let $n\in \mathbb{N}$ and $X_1, \ldots, X_n, Y$ be linearly stable sequence classes such that $X_1(\mathbb{K}) \cdots X_n(\mathbb{K}) \stackrel{1}{\hookrightarrow} Y(\mathbb{K})$. Then $({\cal L}_{X_1, \ldots, X_n;Y}, \|\cdot\|_{X_1, \ldots, X_n;Y})$ is a Banach ideal of $n$-linear operators.
\end{theorem}

\begin{proof} The equality $\widehat{A + \lambda B} = \widehat{A} + \lambda \widehat{B}$ shows that ${\cal L}_{X_1, \ldots, X_n;Y}(E_1, \ldots, E_n;F)$ is a linear subspace of  ${\cal L}(E_1, \ldots, E_n;F)$ and proves the norm axioms for  $ \|\cdot\|_{X_1, \ldots, X_n;Y}$, except that $\|A\|_{X_1, \ldots, X_n;Y} =0 \Longrightarrow A = 0$. This remaining axiom follows from $c_{00}(\cdot) \subseteq X_m(\cdot)$.

To prove that ${\cal L}_{X_1, \ldots, X_n;Y}$ contains the multilinear operators of finite type, it is enough to prove that, if $\varphi_j \in E_j'$, $j = 1, \ldots, n$, and $b \in F$, then $\varphi_1 \otimes \cdots \otimes \varphi_n \otimes b$ is $(X_1, \ldots, X_n;Y)$-summing. Given sequences $(x_j^m)_{j=1}^\infty \in X_m(E_m)$, $m = 1, \ldots,n$, since each $X_m$ is linearly stable, we have $(\varphi_m(x_j^m))_{j=1}^\infty \in X_m(\mathbb{K}).$ Since $X_1(\mathbb{K}) \cdots X_n(\mathbb{K}) \stackrel{1}{\hookrightarrow} Y(\mathbb{K})$, 
$$\left(\varphi_1(x_j^1) \cdots \varphi_n(x_j^n) \right)_{j=1}^\infty \in Y(\mathbb{K}).$$
Considering the linear operator $\lambda \in \mathbb{K} \mapsto \lambda b \in F$ and using the linear stability of $Y$, we get
$$(\varphi_1 \otimes \cdots \otimes \varphi_n \otimes b(x_j^1, \ldots, x_j^n))_{j=1}^\infty=\left(\varphi_1(x_j^1) \cdots \varphi_n(x_j^n)b \right)_{j=1}^\infty \in Y(F), $$
proving that $\varphi_1 \otimes \cdots \otimes \varphi_n \otimes b \in {\cal L}_{X_1, \ldots, X_n;Y}(E_1, \ldots, E_n;F)$.

The assumption $X_1(\mathbb{K}) \cdots X_n(\mathbb{K})\stackrel{1}{\hookrightarrow} Y(\mathbb{K})$ also gives
$$\|I_n\|_{X_1, \ldots, X_n;Y} = \|\widehat{I_n} \colon X_1(\mathbb{K}) \times \cdots \times X_n(\mathbb{K}) \longrightarrow Y(\mathbb{K})\| \leq 1.$$
Since $\|e_1\|_{X_m(\mathbb{K})}= \|e_1\|_{Y(\mathbb{K})}$ for $m = 1, \ldots, n$, and $\widehat{I_n}(e_1, \ldots, e_1) = e_1$, we get $\|I_n\|_{X_1, \ldots, X_n;Y} =1$.

To prove the $n$-ideal property, let $A \in {\cal L}_{X_1, \ldots, X_n;Y}(E_1, \ldots, E_n;F)$, $u_m \in {\cal L}(G_m;E_m)$, $m = 1, \ldots,n$, and $v \in {\cal L}(F;H)$ be given. The linear stability of each $X_m$ and $Y$ together with the $(X_1, \ldots, X_n;Y)$-summability of $A$ give the $(X_1, \ldots, X_n;Y)$-summability of $v \circ A \circ (u_1, \ldots, u_n)$. Moreover,
\begin{align*}\|v \circ A \circ (u_1,& \ldots, u_n) \|_{X_1, \ldots, X_n;Y} \\&= \|({v \circ A \circ (u_1, \ldots, u_n)})^{\wedge}\colon X_1(G_1) \times \cdots \times X_n(G_n) \rightarrow Y(H) \|\\
& = \|{\widehat{v} \circ \widehat{A} \circ (\widehat{u_1}, \ldots, \widehat{u_n})}\colon X_1(G_1) \times \cdots \times X_n(G_n) \rightarrow Y(H) \|\\
& \leq \|\widehat{v}\| \cdot \|\widehat{A}\| \cdot \|\widehat{u_1}\| \cdots \|\widehat{u_n}\| = \|v\| \cdot \|A\|_{X_1, \ldots, X_n;Y} \cdot \|u_1\| \cdots \|u_n\|,
\end{align*}
where $\widehat{v} \colon Y(F) \rightarrow Y(G)$, $\widehat{A} \colon X_1(E_1) \times \cdots \times X_n(E_n) \rightarrow Y(F)$ and $\widehat{u_m} \colon X(G_m) \rightarrow X(E_m)$.

Let $(A_k)_{k=1}^\infty$ be a Cauchy sequence in ${\cal L}_{X_1, \ldots, X_n;Y}(E_1, \ldots, E_n;F)$. From Lemma \ref{lema1}(c) we know that $(A_k)_{k=1}^\infty$ is Cauchy in ${\cal L}(E_1, \ldots, E_n;F)$ too, so there is $A \in {\cal L}(E_1, \ldots, E_n;F)$ such that $A_k \stackrel{\|\cdot\|}{\longrightarrow} A$. The induced maps
$$\widetilde{A_k}, \widetilde{A} \colon \ell_\infty(E_1) \times \cdots \times \ell_\infty(E_n) \longrightarrow \ell_\infty(F), $$
$k \in \mathbb{N}$, are well-defined, $n$-linear and continuous. From
$$\|\widetilde{A_k} - \widetilde{A}\|_{{\cal L}(\ell_\infty(E_1), \ldots , \ell_\infty(E_n); \ell_\infty(F))} = \|\widetilde{A_k - A}\|_{{\cal L}(\ell_\infty(E_1), \ldots , \ell_\infty(E_n); \ell_\infty(F))} = \|A_k - A\|_{{\cal L}(E_1, \ldots, E_n;F)}, $$
we conclude that $\widetilde{A_k} \longrightarrow \widetilde{A}$ in ${\cal L}(\ell_\infty(E_1), \ldots , \ell_\infty(E_n); \ell_\infty(F)). $ For each $k \in \mathbb{N}$, since $A_k$ is $(X_1, \ldots, X_n;Y)$-summing, the map
$$\widehat{A_k} \colon X_1(E_1) \times \cdots \times X_n(E_n) \longrightarrow Y(F), $$
is well-defined, $n$-linear and $\|\widehat{A_k}\|_{{\cal L}(X_1(E_1), \ldots , X_n(E_n); Y(F))} = \|A_k\|_{X_1, \ldots, X_n;Y}.$ Hence $(\widehat{A_k})_{k=1}^\infty$ is a Cauchy sequence in the Banach space ${\cal L}(X_1(E_1), \ldots , X_n(E_n); Y(F))$ (remember that $Y(F)$ is a Banach space). Let $T \in {\cal L}(X_1(E_1), \ldots , X_n(E_n); Y(F))$ be such that $\widehat{A_k} \longrightarrow T$ in ${\cal L}(X_1(E_1), \ldots , X_n(E_n); Y(F))$. Let $(x_j^m)_{j=1}^\infty \in X_m(E_m)$, $m = 1, \ldots, n$, be given. Of course
$$\widetilde{A}\left((x_j^1)_{j=1}^\infty, \ldots, (x_j^n)_{j=1}^\infty \right)= \widehat{A}\left((x_j^1)_{j=1}^\infty, \ldots, (x_j^n)_{j=1}^\infty \right).  $$
On the one hand, from $\widetilde{A_k} \longrightarrow \widetilde{A}$ we have $$\widetilde{A_k}\left((x_j^1)_{j=1}^\infty, \ldots, (x_j^n)_{j=1}^\infty  \right) \longrightarrow \widetilde{A}\left((x_j^1)_{j=1}^\infty, \ldots, (x_j^n)_{j=1}^\infty  \right) {\rm ~in~} \ell_\infty(F).$$
On the other hand, from $\widehat{A_k} \longrightarrow T$ we have
$$\widehat{A_k}\left((x_j^1)_{j=1}^\infty, \ldots, (x_j^n)_{j=1}^\infty  \right) \longrightarrow T\left((x_j^1)_{j=1}^\infty, \ldots, (x_j^n)_{j=1}^\infty  \right) {\rm ~in~} Y(F), $$
and since $Y(F) \hookrightarrow \ell_\infty(F)$ is a norm 1 embedding, it follows that
$$\widehat{A_k}\left((x_j^1)_{j=1}^\infty, \ldots, (x_j^n)_{j=1}^\infty  \right) \longrightarrow T\left((x_j^1)_{j=1}^\infty, \ldots, (x_j^n)_{j=1}^\infty  \right) {\rm ~in~} \ell_\infty(F). $$
This gives $\widetilde{A}\left((x_j^1)_{j=1}^\infty, \ldots, (x_j^n)_{j=1}^\infty  \right) = T\left((x_j^1)_{j=1}^\infty, \ldots, (x_j^n)_{j=1}^\infty  \right)\in Y(F) $, proving that $A$ is $(X_1, \ldots, X_n;Y)$-summing. Finally,
\begin{align*}\|A_k - A\|_{X_1,\ldots, X_n;Y}& = \|\widehat{A_k - A}\|_{{\cal L}(X_1(E_1), \ldots , X_n(E_n); Y(F))} =  \|\widehat{A_k} -\widehat{ A}\|_{{\cal L}(X_1(E_1), \ldots , X_n(E_n); Y(F))}\\
&= \|\widehat{A_k} -T\|_{{\cal L}(X_1(E_1), \ldots , X_n(E_n); Y(F))} \longrightarrow 0,
\end{align*}
what proves that $A_k \longrightarrow A$ in ${\cal L}_{X_1, \ldots, X_n;Y}(E_1, \ldots, E_n;F)$.
\end{proof}

\begin{remark}\label{errodiana}\rm A related result on the generation of multilinear ideals based on the transformation of vector-valued sequences can be found in \cite[Theorem 3]{diana}. An important assumption is missing there: for \cite[Theorem 3]{diana} to be true, one should assume -- as we have just done in Theorem \ref{propideal} -- that the underlying sequence spaces are linearly stable.
\end{remark}

\begin{example}\label{exam} \rm Important ideals of linear and multilinear operators that have been studied individually in the literature are particular cases of the ideals generated by Theorem \ref{propideal}. In Example \ref{erro} we saw that, for every $n \in \mathbb{N}$, the ideal of almost summing $n$-linear operators coincides with
${\cal L}_{n,\ell_2^u(\cdot); {\rm Rad}(\cdot)} = {\cal L}_{n,\ell_2^w(\cdot); {\rm RAD}(\cdot)}.  $ This ideal is generated by Theorem \ref{propideal} because $\ell_{2} \cdots \ell_{2} \stackrel{1}{\hookrightarrow} \ell_2 = {\rm Rad}(\mathbb{K})$.
We add just three illustrative examples:\\
(a) The class ${\cal L}^n_{wsc,0}$ of $n$-linear operators that are weakly sequentially continuous at the origin (a more appropriate term would be {\it sequentially weak-to-norm continuous at the origin}), has played an important role in the study of spaces of multilinear operators and homogeneous polynomials. It has proved to be very useful in the study of the reflexivity of spaces of multilinear and polynomial operators (see, e.g., \cite{aad, arondineen, livrodineen, valdivia}). Considering the linearly stable sequence classes $c_0^w(\cdot)$ and $c_0(\cdot)$, as $c_0 \stackrel{(n)}{\cdots} c_0 \stackrel{1}{\hookrightarrow} c_0$ for every $n$, we have
$${\cal L}^n_{wsc,0}= {\cal L}_{n,c_0^w(\cdot);c_0(\cdot)}.$$
In the linear case, we recover the closed ideal of completely continuous operators $\cal CC$, that is, ${\cal L}^1_{wsc,0}= {\cal CC}$.\\
(b) Let $p,p_1, \ldots, p_n \geq 1$ be such that $\frac1p = \frac{1}{p_1} + \cdots +\frac{1}{p_n}$. The class ${\cal L}_{d;p_1, \ldots,p_n}$ of $(p_1, \ldots, p_n)$-dominated multilinear operators is the oldest and one of the most studied multilinear/
polynomial generalizations of the ideal of absolutely $p$-summing linear operators. It was introduced by Pietsch \cite{pietsch} and its theory has been developed by several authors, e.g. \cite{boperu, cg, david, popa}. Considering the linearly stable sequence classes $\ell_p(\cdot)$ and $\ell_{p_j}^w(\cdot)$, as $\ell_{p_1} \cdots \ell_{p_n} \stackrel{1}{\hookrightarrow} \ell_p$, we have
$${\cal L}_{d;p_1, \ldots,p_n} = {\cal L}_{\ell_{p_1}^w(\cdot), \ldots,\ell_{p_n}^w(\cdot); \ell_p(\cdot) }. $$
Developing the theory considering quasi-Banach sequence spaces in the sense of Remark \ref{remark1} {\rm (b)}, the general case of dominated multilinear operators, that is, ${\cal L}_{d;p_1, \ldots,p_n}$ for $p,p_1, \ldots, p_n > 0$ with $\frac1p = \frac{1}{p_1} + \cdots +\frac{1}{p_n}$, would be accomodated as well.\\
(c) Reasoning as in \cite[Proposition 3.2]{jamilson}, it is easy to see that a linear or multilinear operator is Cohen almost summing in the sense of \cite{bu12} if and only if it is $({\rm Rad}(\cdot), \ldots, {\rm Rad}(\cdot); \ell_2\langle\,\cdot\,\rangle)$-summing.
\end{example}
%

Let us see that the ideals generated by Theorem \ref{propideal} enjoy a distinguished property. Let $(P_m)_{m=1}^\infty$ be canonical projections associated to a Schauder basis of an infinite dimensional Banach space $E$, let $id_E$ be the identity operator on $E$ and let $\cal K$ denote the closed ideal of compact operators. Then $P_m \in {\cal K}(E;E)$ for every $m$, $P_m(x) \longrightarrow id_E(x)$ for every $x \in E$, $\sup_m \|P_m\| < +\infty$, but $id_E \notin {\cal K}(E;E)$. Next we prove that this does not occur with the ideals generated by Theorem \ref{propideal}:
%

\begin{proposition}\label{propried} Let $n \in \mathbb{N}$ and $X_1, \ldots, X_n,Y$ be finitely determined sequence classes satisfying the conditions of Theorem \ref{propideal}, let $(A_m)_{m=1}^\infty$ be a sequence in ${\cal L}_{X_1, \ldots, X_n;Y}(E_1, \ldots, E_n;F)$ and $A \in {\cal L}(E_1,\ldots,E_n;F)$. If $\sup_m \|A_m\|_{X_1, \ldots, X_n;Y} < +\infty$ and
\begin{equation*}A_m(x_1,\ldots,x_n) \longrightarrow A(x_1,\ldots,x_n)\ {\rm in}\ F {\rm ~for~every~} (x_1,\ldots,x_n) \in E_1 \times \cdots \times E_n,
\end{equation*}
then $A \in {\cal L}_{X_1, \ldots, X_n;Y}(E_1, \ldots, E_n;F)$ and $\|A\|_{X_1, \ldots, X_n;Y} \leq \sup_m \|A_m\|_{X_1, \ldots, X_n;Y}$.
\end{proposition}

\begin{proof} Using Lemma \ref{lema1}(b), for any $k \in \mathbb{N}$,
\begin{align*}
\left\Vert (A(x_j^1, \ldots, x_j^n))_{j=1}^k\right\Vert_{Y(F)} & = \left\Vert (A(x_1^1, \ldots, x_1^n), \ldots, A(x_k^1, \ldots, x_k^n), 0, 0, \ldots)_{j=1}^k\right\Vert_{Y(F)}\\
& = \left\Vert (\lim_mA_m(x_1^1, \ldots, x_1^n), \ldots, \lim_mA_m(x_k^1, \ldots, x_k^n), 0, 0, \ldots)_{j=1}^k\right\Vert_{Y(F)}\\
& = \left\Vert \lim_m(A_m(x_1^1, \ldots, x_1^n), \ldots, A_m(x_k^1, \ldots, x_k^n), 0, 0, \ldots)_{j=1}^k\right\Vert_{Y(F)}\\
&= \lim_m \left\Vert (A_m(x_j^1, \ldots, x_j^n))_{j=1}^k\right\Vert_{Y(F)}\\
& \le \sup_m \|A_m\|_{X_1, \ldots, X_n;Y}\cdot \prod_{i=1}^n \left\Vert (x_j^i)_{j=1}^k \right\Vert_{X(E_i)}.
\end{align*}
Now the result follows from Proposition \ref{firstprop}.
\end{proof}

\begin{remark} \rm In Proposition \ref{propried}, the assumption that the sequence classes are finitely determined cannot be dropped. Indeed, denoting by $(P_m)_{m=1}^\infty$ the canonical projections associated to a Schauder basis of a reflexive infinite dimensional Banach space  $E$, we have  $(P_m)_{m=1}^\infty \subseteq {\cal CC}(E; E) = {\cal L}_{c_0^w(\cdot);c_0(\cdot)}(E; E)$ (cf. Example \ref{exam}(a)), $\sup_m \|P_m\| < +\infty$,  and $P_m(x) \longrightarrow id_{E}(x)$ for every $ x \in E$, but
 $id_{E} \notin {\cal CC}(E;E)$. This happens because $c_0^w(\cdot)$ and $c_0(\cdot)$ are not finitely determined.
\end{remark}

\section{Multilinear stability}
The purpose of this section is to establish the strong contrast between the linear and multilinear cases with respect to the preservation of vector-valued sequences. Two main differences arise: linear stability does not imply multilinear stability and the proofs of multilinear stability use {\it ad hoc} arguments to each case. As an application of our multilinear stability results we improve a result of \cite{bu12}.

\begin{definition}\rm A sequence class $X$ is said to be {\it multilinearly stable} if ${\cal L}_{n,X;X} \stackrel{1}{=} {\cal L}^n$ for every $n$, that is, $(A(x_j^1, \ldots, x_j^n))_{j=1}^\infty \in X(F)$ whenever $(x_j^m)_{j=1}^\infty \in X(E_m)$, $m = 1, \ldots, n$, and
$$\|\widehat{A}\colon X(E_1)\times \cdots \times X(E_n) \longrightarrow X(F)\| = \|A\|  $$
for all Banach spaces $E_1, \ldots, E_n$ and $F$ and operators $A \in {\cal L}(E_1, \ldots, E_n;F)$.
\end{definition}

\begin{example}\label{nex}\rm (a) It is easy to check that the sequence classes $\ell_\infty(\cdot)$, $c_0(\cdot)$ and $\ell_p(\cdot)$, $1 \leq p \leq + \infty$, are multilinearly stable (see (\ref{holder})).\\
(b) Consider the bilinear operator
$$A \colon \ell_2 \times \ell_2 \longrightarrow \ell_1~,~A\left((x_j)_{j=1}^\infty, (y_j)_{j=1}^\infty \right)= (x_jy_j)_{j=1}^\infty.$$
Since $(e_k)_{k=1}^\infty$ is weakly null in $\ell_2$ but not in $\ell_1$, the sequence class $c_0^w(\cdot)$ is not multilinearly stable.
\end{example}

We have just seen that multilinear stability does not follow from linear stability ($c_0^w(\cdot)$ is linearly stable but not multilinearly stable). Let us see a remarkable example concerning summable sequences: on the one hand, the sequence classes $\ell_p^w(\cdot)$ and $\ell_p^u(\cdot)$ are linearly stable for every $1 \leq p < + \infty$. On the other hand:

\begin{theorem}\label{firsttheo} Let $1 \leq p < +\infty$. The sequence classes $\ell_p^w(\cdot)$ and $\ell_p^u(\cdot)$ are multilinearly stable if and only if $p = 1$.
\end{theorem}

\begin{proof} Let $A \in {\cal L}(E_1, \ldots, E_n;F)$ and $(x_j^m)_{j=1}^\infty \in \ell_1^w(E_m)$, $m = 1, \ldots, n$, be given. We shall use the following trick: for every $k \in \mathbb{N}$, 
\begin{align*}&\sum_{j=1}^k A(x_j^1, \ldots, x_j^n)\\
& = \int_0^1 \cdots \int_0^1 A\left(\sum_{j=1}^kr_j(t_1)x_j^1, \ldots, \sum_{j=1}^kr_j(t_{n-1})x_j^{n-1}, \sum_{j=1}^k \prod_{l=1}^{n-1}r_j(t_l)x_j^n \right)dt_1 \cdots dt_{n-1}.
 \end{align*}

For all $|\lambda_j^i| \leq 1$, $j = 1, \ldots, k$, $i = 1, \ldots, n$, we have
\begin{align*}
&\sum_{j=1}^k \lambda_j^1 \cdots \lambda_j^n A(x_j^1, \ldots, x_j^n) = \sum_{j=1}^k A(\lambda_j^1x_j^1, \ldots, \lambda_j^nx_j^n)\\
&= \int_0^1 \cdots \int_0^1 A\left(\sum_{j=1}^kr_j(t_1)\lambda_j^1x_j^1, \ldots, \sum_{j=1}^kr_j(t_{n-1})\lambda_j^{n-1}x_j^{n-1}, \sum_{j=1}^k \prod_{l=1}^{n-1}r_j(t_l)\lambda_j^nx_j^n \right)dt_1 \cdots dt_{n-1},
\end{align*}
so
\begin{align*}& \left\|\sum_{j=1}^k  \lambda_j^1 \cdots \lambda_j^n A(x_j^1, \ldots, x_j^n) \right\|\\
& \leq \sup_{t_i \in [0,1]}\left\| A\left(\sum_{j=1}^kr_j(t_1)\lambda_j^1x_j^1, \ldots, \sum_{j=1}^kr_j(t_{n-1})\lambda_j^{n-1}x_j^{n-1}, \sum_{j=1}^k \prod_{l=1}^{n-1}r_j(t_l)\lambda_j^nx_j^n \right) \right\|\\
&\leq \|A\|\cdot\sup_{t_i \in [0,1]}\left\| \sum_{j=1}^kr_j(t_1)\lambda_j^1x_j^1\right\|\cdots \left\|\sum_{j=1}^kr_j(t_{n-1})\lambda_j^{n-1}x_j^{n-1}\right\|\cdot \left\| \sum_{j=1}^k \prod_{l=1}^{n-1}r_j(t_l)\lambda_j^nx_j^n  \right\|\\
&\leq \|A\|\cdot  \prod_{m=1}^n \left\|(x_j^m)_{j=1}^k\right\|_{w,1},
\end{align*}
for every $k \in \mathbb{N}$. Therefore,
\begin{align*} \|(A(x_j^1, \ldots, &x_j^n)_{j=1}^k\|_{w,1}= \sup\left\{\left\|\sum_{j=1}^k \lambda_j A(x_j^1, \ldots, x_j^n) \right\| : |\lambda_j| \leq 1, j=1, \ldots, n \right\}\\
&= \sup\left\{\left\|\sum_{j=1}^k \lambda_j^1 \cdots \lambda_j^n A(x_j^1, \ldots, x_j^n) \right\| : |\lambda_j^i| \leq 1, j = 1, \ldots, k, i = 1, \ldots, n \right\}\\
&\leq \|A\|\cdot  \prod_{m=1}^n \left\|(x_j^m)_{j=1}^k\right\|_{w,1}\leq \|A\|\cdot  \prod_{m=1}^n \left\|(x_j^m)_{j=1}^\infty\right\|_{w,1},
\end{align*}
for every $k \in \mathbb{N}$. As $\ell_1^w(\cdot)$ is finitely determined, by Proposition \ref{firstprop} we conclude that $\widehat{A} \colon \ell_1^w(E_1) \times \cdots \times \ell_1^w(E_n) \longrightarrow \ell_1^w(F)$ is well-defined, $n$-linear,  continuous and that $\|\widehat{A}\| \leq \|A\|$. As $\ell_1^w(\cdot)$ is linearly stable, the reverse inequality follows from Lemma \ref{lema1}(c), completing the proof that $\ell_1^w(\cdot)$ is multilinearly stable.

As $\ell_1^u(\cdot) < \ell_1^w(\cdot)$ and $\ell_1^w(\cdot)$ is finitely determined, by Corollary \ref{cornew}(i) we have
$${\cal L}_{n,\ell_1^u(\cdot); \ell_1^u(\cdot)} \stackrel{1}{=} {\cal L}_{n,\ell_1^w(\cdot); \ell_1^w(\cdot)} \stackrel{1}{=} {\cal L}^n, $$
for every $n$, what gives the multilinear stability of $\ell_1^u(\cdot)$.

Given $1 < p < +\infty$, choose $n \geq p^*$, $\frac{1}{p} + \frac{1}{p^*} = 1$, and consider the continuous $n$-linear operator
$$A \colon \left(\ell_{p^*}\right)^n \longrightarrow \ell_1~,~A\left((\lambda_j^1)_{j=1}^\infty, \ldots, (\lambda_j^n)_{j=1}^\infty \right) = (\lambda_j^1\cdots \lambda_j^n)_{j=1}^\infty . $$
As $(e_k)_{k=1}^\infty $ belongs to $\ell_p^w(\ell_{p^*})$ but not to $\ell_p^w(\ell_{1})$, the sequence class $\ell_p^w(\cdot)$ fails to be multilinearly stable. As $\ell_p^u(\cdot) < \ell_p^w(\cdot)$ and $\ell_p^w(\cdot)$ is finitely determined, by Corollary \ref{cornew}(i) we have
$${\cal L}_{n,\ell_p^u(\cdot); \ell_p^u(\cdot)} \stackrel{1}{=} {\cal L}_{n,\ell_p^w(\cdot); \ell_p^w(\cdot)} {\neq} {\cal L}^n, $$
proving that $\ell_p^u(\cdot)$ is not multilinearly stable either.
\end{proof}

\begin{remark}\rm (a) The complex case of $\ell_p^w(\cdot)$ in Theorem \ref{firsttheo} was treated in the unpublished thesis \cite{tesecarlos}, but the argument used there does not work in the real case and the case of $\ell_p^u(\cdot)$ is not touched.\\
(b) The proof of Theorem \ref{firsttheo} actually proves that ${\cal L}_{\ell_1^{t_1}(\cdot),  \ldots, \ell_1^{t_n}(\cdot); \ell_1^u(\cdot)} \stackrel{1}{=} {\cal L}^n$ with $t_j = u$ for some $j \in \{1,\ldots,n\}$ and $t_i = w$ for $i \neq j$.
\end{remark}

As noted in the Introduction and according to what we have just done in the proof of Theorem \ref{firsttheo}, typical multilinear arguments are needed to prove the multilinear stability of a given sequence class. Here is another example, concerning now sequence classes that play an important role in the geometry and in the probability of Banach spaces:

\begin{theorem}\label{theorad} The sequence classes ${\rm Rad}(\cdot)$ and ${\rm RAD}(\cdot)$ are multilinearly stable.
\end{theorem}

\begin{proof} Let $B \in {\cal L}(E_1, \ldots, E_n;F)$. By $B_L$ we denote the linearization of $B$ from the completed projective tensor product $E_1 \widehat\otimes_\pi \cdots \widehat\otimes_\pi E_n$ to $F$. The projective norm is denoted by $\| \cdot \|_\pi$. For any natural number $k$ and any $x_j^i \in E_i$, $i=1,\ldots,n$, $j = 1, \ldots, k$,
\begin{align*}
\left\|(B(x_j^1, \ldots, x_j^n))_{j=1}^k\right\|_{\mathrm{Rad}(F)} & =
  \left(\int_0^1 \left\Vert B_L\left(\sum_{j=1}^k r_j(t)x_j^{(1)}\otimes \dots \otimes x_j^{(n)}\right) \right\Vert^2 dt\right)^{1/2}\\
& \leq \|B\| \cdot \left(\int_0^1 \left\Vert \sum_{j=1}^k r_j(t)x_j^{(1)}\otimes \dots \otimes x_j^{(n)} \right\Vert^2_\pi dt\right)^{1/2}\\
& \leq \|B\| \cdot \sup_{t \in [0,1]} \left\Vert \sum_{j=1}^k r_j(t)x_j^{(1)}\otimes \dots \otimes x_j^{(n)} \right\Vert_\pi\\
& \overset{(*)}{=} \|B\|\cdot \sup_{t \in [0,1]} \sup_{A \in B_{\mathcal{L}(E_1,\ldots, E_n)}} \left\vert \sum_{j=1}^k A\left(r_j(t)x_j^{(1)},\ldots,x_j^{(n)}\right) \right\vert\\
& \leq \|B\|\cdot \sup_{t \in [0,1]} \sup_{A \in B_{\mathcal{L}(E_1,\ldots, E_n)}} \sum_{j=1}^k |r_j(t)|\cdot \left\vert A\left(x_j^{(1)},\ldots,x_j^{(n)}\right) \right\vert\\
&= \|B\|\cdot \sup_{A \in B_{\mathcal{L}(E_1,\ldots, E_n)}} \sum_{j=1}^k \left\vert A\left(x_j^{(1)},\ldots,x_j^{(n)}\right) \right\vert\\
& \overset{(**)}{\leq} \|B\| \cdot \sup_{A \in B_{\mathcal{L}(E_1,\ldots, E_n)}} \|A\|\cdot \prod_{m=1}^n \left\Vert (x_j^{m})_{j=1}^k \right\Vert_{\mathrm{Rad}(E_i)}\\
&= \|B\| \cdot \prod_{m=1}^n \left\Vert (x_j^{m})_{j=1}^k \right\Vert_{\mathrm{Rad}(E_i)},
\end{align*}
where $(*)$ follows from the dual formula for the projective norm \cite[Formula (2.3), p.\,23]{ryan}, and $(**)$ follows from \cite[Proposition 3.1]{BotelhoBlasco}. As ${\rm RAD}(\cdot)$ is finitely determined, the calculation above says that the equivalent conditions of Proposition \ref{firstprop}, with $X_1 = \cdots = X_n = Y = {\rm RAD}(\cdot)$, hold for $B$ and that $\|\widehat{B}\| \leq \|B\|$. The reverse inequality follows from Lemma \ref{lema1}(c). This proves that ${\rm RAD}(\cdot)$ is multlinearly stable.

The completeness of $L_2([0,1];F)$ tells us that ${\rm Rad}(\cdot) \prec {\rm RAD}(\cdot)$. Since ${\rm RAD}(\cdot)$ is finitely determined, by Corollary \ref{cornew}(ii) we have
$${\cal L}_{n,{\rm Rad}(\cdot); {\rm Rad}(\cdot)} \stackrel{1}{=} {\cal L}_{n,{\rm RAD}(\cdot); {\rm RAD}(\cdot)} \stackrel{1}{=} {\cal L}^n, $$
for every $n$, proving the multilinear stability of ${\rm Rad}(\cdot)$.
\end{proof}

The sequence class $\ell_p\langle \,\cdot\, \rangle$ was introduced by Cohen \cite{cohen73} in order to describe the linear operators having absolutely $p^{*}$-summing adjoints, $\frac{1}{p} + \frac{1}{p^*} = 1$. Much research has been made since then using these spaces, see, e.g., \cite{aywa,BotelhoBlasco,budiestel,bu12,jamilson13}. Now we show that $\ell_p\langle \,\cdot\, \rangle$ is multilinearly stable.

\begin{theorem}\label{lpCohen} Let $p, p_1, \dots, p_n \geq 1$ be such that $\frac{1}{p} \leq \frac{1}{p_1} + \cdots + \frac{1}{p_n}$. Then $${\cal L}_{\ell_{p_1}\langle\, \cdot\, \rangle, \ldots,\ell_{p_n}\langle\,\cdot\,\rangle; \ell_p\langle\,\cdot\,\rangle } \stackrel{1}{=} {\cal L}^n.$$
In particular, the sequence class $\ell_p\langle\,\cdot\,\rangle$ is multilinearly stable for every $1 \leq p < +\infty$.
\end{theorem}

\begin{proof} Let us prove the bilinear case. It is well known that $\ell_q\langle E \rangle =\ell_q \widehat{\otimes}_\pi E$ isometrically for all $q \geq 1$ and Banach space $E$ (see, e.g., \cite[Corollary 3.9]{aywa} or \cite[Theorem 1]{budiestel}). 

Let $\varepsilon > 0$. Given $z \in \ell_{p_1} \otimes_\pi E$ and $w \in \ell_{p_2} \otimes_\pi F$, we can take representations
$$z = \sum_{j=1}^k (\lambda_j^i)_{i=1}^\infty \otimes x_j = \sum_{j=1}^k (\lambda_j^i x_j)_{i=1}^\infty = \left(\sum_{j=1}^k\lambda_j^i x_j \right)_{i=1}^\infty, $$
$$w = \sum_{l=1}^m (\alpha_l^i)_{i=1}^\infty \otimes y_l = \sum_{l=1}^m (\alpha_l^i y_l)_{i=1}^\infty = \left(\sum_{l=1}^m\alpha_l^i y_l \right)_{i=1}^\infty, $$
with $x_1, \ldots, x_k \in E,y_1, \ldots, y_m \in F$, $(\lambda_j^i)_{i=1}^\infty \in \ell_{p_1}$ for $j = 1,\ldots, k$, $(\alpha_l^i)_{i=1}^\infty \in \ell_{p_2}$ for $ l = 1, \ldots, m$, and
$$\sum_{j=1}^k \left\|(\lambda_j^i)_{i=1}^\infty \right\|_{p_1}\cdot \|x_j\| = \sum_{j=1}^k\left(\sum_{i=1}^\infty |\lambda_j^i|^{p_1} \cdot \|x_j\|^{p_1}\right)^{1/{p_1}} < (1+\varepsilon)\left\|x \right\|_\pi,$$
$$\sum_{l=1}^m \left\|(\alpha_l^i)_{i=1}^\infty \right\|_{p_2}\cdot \|y_l\| = \sum_{l=1}^m\left(\sum_{i=1}^\infty |\alpha_l^i|^{p_2} \cdot \|y_l\|^{p_2}\right)^{1/{p_2}} < (1+\varepsilon)\left\|y \right\|_\pi.$$
From $\frac{1}{p} \leq \frac{1}{p_1} + \frac{1}{p_2}$ we conclude that $\left(\lambda_j^i \alpha_l^i \right)_{i=1}^\infty \in \ell_p$  for $j = 1,\ldots, k, l = 1, \ldots, m$. The bilinearity of $A$ yields
\begin{align*}  
\left( A \left(\sum_{j=1}^k\lambda_j^i x_j , \sum_{l=1}^m\alpha_l^i y_l\right)\right)_{i=1}^\infty& = \left( \sum_{j=1}^k \sum_{l=1}^m \lambda_j^i \alpha_l^i A(x_j,y_l)\right)_{i=1}^\infty =  \sum_{j=1}^k \sum_{l=1}^m \left(\lambda_j^i \alpha_l^i A(x_j,y_l)\right)_{i=1}^\infty\\
& =  \sum_{j=1}^k \sum_{l=1}^m \left(\lambda_j^i \alpha_l^i \right)_{i=1}^\infty \otimes A(x_j,y_l) \in \ell_p \widehat{\otimes}_\pi G.
\end{align*}
This shows that the operator
$$\overline{A} \colon \ell_{p_1} \otimes_\pi E \times \ell_{p_2} \otimes_\pi F  \longrightarrow \ell_p \widehat{\otimes}_\pi G~,~\overline{A}(z,w) = \left(A(z_i, w_i) \right)_{i=1}^\infty, 
$$
where $z = (z_i)_{i=1}^\infty$ and $w = (w_i)_{i=1}^\infty$, is well defined. It is plain that $\overline{A}$ is bilinear. Moreover,
\begin{align*} \left\| \overline{A}(z,w) \right\|_\pi& \leq \sum_{j=1}^k \sum_{l=1}^m \left\|\left(\lambda_j^i \alpha_l^i \right)_{i=1}^\infty\right\|_p \cdot \|A(x_j,y_l)\| \\
&\leq \|A\| \cdot \sum_{j=1}^k \sum_{l=1}^m \left[ \left(\sum_{i=1}^\infty |\lambda_j^i|^{p_1}\right)^{1/{p_1}}\cdot \left(\sum_{i=1}^\infty|\alpha_l^i|^{p_2}\right)^{1/{p_2}} \right] \cdot \|x_j\| \cdot \|y_l\|\\
&= \|A\| \cdot\left[ \sum_{j=1}^k  \left(\sum_{i=1}^\infty |\lambda_j^i|^{p_1}\right)^{1/{p_1}}\cdot \|x_j\| \right]\cdot \left[ \sum_{l=1}^m \cdot \left(\sum_{i=1}^\infty|\alpha_l^i|^{p_2}\right)^{1/{p_2}}   \cdot \|y_l\|\right]\\
&   < (1+\varepsilon)^2\|A\| \cdot \left\|z \right\|_\pi \cdot \|w\|_\pi.
\end{align*}
Making $\varepsilon \longrightarrow 0^+$ we conclude that the bilinear operator $\overline{A}$
is continuous and $\|\overline{A}\| \leq \|A\|$. Reasoning as in Lemma \ref{lema1} we get $\|\overline{A}\| = \|A\|$. Since multilinear operators are uniformly continuous on bounded sets and $\ell_p \widehat{\otimes}_\pi G$ is a Banach space, we  can consider
$$\widetilde{A} \colon \ell_{p_1} \widehat{\otimes}_\pi E \times \ell_{p_2} \widehat{\otimes}_\pi F  \longrightarrow \ell_p \widehat{\otimes}_\pi G $$
the unique extension of $\overline{A}$ with $\|\overline{A}\| = \|\widetilde{A}\| = \|A\|$. From (\ref{holder}), the bilinear operator
$$\widehat{A} \colon \ell_{p_1}(E) \times \ell_{p_2}(F)\longrightarrow \ell_p(G)~,~\widehat{A}\left((x_j)_{j=1}^\infty , (y_j)_{j=1}^\infty\right) =   \left(A(x_j,y_j) \right)_{j=1}^\infty,$$
is continuous. Given $z \in  \ell_{p_1} \widehat{\otimes}_\pi E$ and $w \in  \ell_{p_2} \widehat{\otimes}_\pi F$, take $(z_k)_{k=1}^\infty \in \ell_{p_1} \otimes_\pi E$ and $(w_k)_{k=1}^\infty \in \ell_{p_2} \otimes_\pi F$ such that $z_k \stackrel{\pi}{\longrightarrow} z$ and $w_k \stackrel{\pi}{\longrightarrow} w$. By the continuity of $\widetilde{A}$, $\widetilde{A}(z_k, w_k) \stackrel{\pi}{\longrightarrow} \widetilde{A}(z,w). $ As  $\ell_p \widehat{\otimes}_\pi G \stackrel{1}{\hookrightarrow} \ell_p(G)$, we have $\widetilde{A}(z_k,w_k) \longrightarrow \widetilde{A}(z,w)$ in $\ell_p(G)$. For the same reason, we have $z_k \longrightarrow z$ in $\ell_{p_1}(E)$ and $w_k \longrightarrow w$ in $\ell_{p_2}(F)$. By the continuity of $\widehat{A}$,
$$\widetilde{A}(z_k,w_k) = \widehat{A}(z_k,w_k) \longrightarrow \widehat{A}(z,w) {\rm ~in~} \ell_p(G), $$
  so $\widehat{A}(z,w)  = \widetilde{A}(z,w) \in \ell_p \widehat{\otimes}_\pi G$ for all $(z,w) \in \ell_{p_1} \widehat{\otimes}_\pi E \times \ell_{p_2} \widehat{\otimes}_\pi F$, completing the proof that $\left(A(x_j,y_j) \right)_{j=1}^\infty \in \ell_p \widehat{\otimes}_\pi G$ whenever $\left(x_j \right)_{j=1}^\infty \in \ell_{p_1} \widehat{\otimes}_\pi E$ and $\left(y_j \right)_{j=1}^\infty \in  \ell_{p_2} \widehat{\otimes}_\pi F$. The norm equality $\|\widetilde{A}\| = \|\overline{A}\| = \|A\|$ completes the proof.
\end{proof}

We finish the paper with a concrete application of the multilinear stability of ${\rm Rad}(\cdot)$ and $\ell_p\langle\,\cdot\,\rangle$. Given an operator ideal $\cal I$, an $n$-linear operator $A$ belongs to ${\cal I} \circ {\cal L}^n$ if $A = u \circ B$ where $u$ is a linear operator belonging to $\cal I$ and $B$ is $n$-linear; and $A$ belongs to ${\cal L}^n \circ {\cal I}$ if $A = B \circ (u_1, \ldots, u_n)$ where each linear operator $u_j$ belongs to $\cal I$ and $B$ is $n$-linear. We denote by ${\cal D}_p^n$ the ideal of Cohen strongly $p$-summing $n$-linear operators (see Introduction (vi)), as usual we write ${\cal D}_p$ in the linear case, and by ${\cal D}_{as}^n$ we denote the ideal of Cohen almost summing $n$-linear operators (see Example \ref{exam}(c)). One of the main results of \cite{bu12}, namely \cite[Theorem 3.4]{bu12}, asserts that ${\cal D}_p^n \subseteq {\cal D}_{as}^n$, a result we improve next.  Moreover, the proof we give for the aforementioned inclusion is shorter than the original proof.

\begin{corollary} For all $p >1$ and $n \in \mathbb{N}$, ${\cal D}_p^n \cup \left({\cal L}^n\circ {\cal D}_p\right) \subseteq {\cal D}_{as}^n$.
\end{corollary}

\begin{proof} By $\pi_{p^*}^{\rm dual}$ we denote the ideal of all linear operators having $p^*$-summing adjoints. In the chain
$${\cal D}_p^n = {\cal D}_p \circ {\cal L}^n = \pi_{p^*}^{\rm dual} \circ {\cal L}^n \subseteq {\cal D}_{as}^n, $$
the first equality follows from \cite[Proposition 3.1]{bu12}, the second from \cite{cohen73} and the inclusion from \cite[Theorem 1]{bukranz} and Theorem \ref{theorad}. And in the chain
$${\cal L}^n\circ {\cal D}_p = {\cal L}^n\circ \pi_{p^*}^{\rm dual} \subseteq {\cal D}_{as}^n,$$
the equality follows from \cite{cohen73} and the inclusion from \cite[Theorem 1]{bukranz} and Theorem \ref{lpCohen}.
\end{proof}

\bigskip

\noindent Faculdade de Matem\'atica~~~~~~~~~~~~~~~~~~~~~~Departamento de Ci\^{e}ncias Exatas\\
Universidade Federal de Uberl\^andia~~~~~~~~ Universidade Federal da Para\'iba\\
38.400-902 -- Uberl\^andia -- Brazil~~~~~~~~~~~~ 58.297-000 -- Rio Tinto -- Brazil\\
e-mail: botelho@ufu.br ~~~~~~~~~~~~~~~~~~~~~~~~~e-mails: jamilson@dcx.ufpb.br,\\ \hspace*{9,3cm}jamilsonrc@gmail.com

\end{document}